\documentclass[letterpaper, 10 pt, conference]{ieeeconf}  
\newif\ifDraft
\Drafttrue

\newcommand{\papertitle}{Distributionally Robust Optimization using Cost-Aware Ambiguity Sets}

\IEEEoverridecommandlockouts                              

\makeatletter
\let\NAT@parse\undefined
\makeatother

\usepackage{mynotation}

\makeatletter
\newcounter{ALC@unique}
\makeatother
\let\labelindent\relax

\usepackage[draft=true,xcolor=false,clevethm=true]{defaultpackages}
\pgfplotsset{compat=1.8}
\usepackage{lipsum}  
\usepackage{mathtools}

\acrodef{ERM}{empirical risk minimization}
\acrodef{DRO}{distributionally robust optimization}
\acrodef{SAA}{sample average approximation}
\acrodef{LP}{linear program}
\acrodef{LPing}[LP]{linear programming}
\newacroindefinite{LP}{an}{a}

\newcommand{\cadro}{\textsc{Cadro}}
\newcommand{\tvdro}{\textsc{Tv-dro}}

\newif\ifDetails
\newif\ifExtended
\newif\ifArxiv
\newif\ifShowillustration
\newif\ifHoeffding

\providecommand\papertitle{Untitled}
\title{\LARGE \bf
\papertitle
}
\author{Mathijs~Schuurmans and Panagiotis~Patrinos
  \thanks{M. Schuurmans and P. Patrinos are with the Department 
  of Electrical Engineering (\textsc{esat-stadius}), KU Leuven, 
  Kasteelpark Arenberg 10, 3001 Leuven, Belgium.
  Email: \texttt{\{mathijs.schuurmans, panos.patrinos\}@esat.kuleuven.be}}
  \thanks{
This work was supported by
the Research Foundation Flanders (FWO) research projects G081222N, G033822N, G0A0920N;
European Union's Horizon 2020 research and innovation programme under the Marie Sk{\l}odowska-Curie grant agreement No. 953348;
Ford KU Leuven Research Alliance Project KUL0075;
}
}

\newcommand{\Vop}{V^{\!\star}}
\newcommand{\Xop}{X^{\!\star}}
\newcommand{\xop}{x^\star}
\newcommand{\trs}{\tau}     
\newcommand{\p}{p^\star}    
\newcommand{\ord}[2]{#1_{(#2)}}
\newcommand{\kthr}{\kappa}
\newcommand{\ph}[1][]{
	\def\tempa{}
	\def\tempb{#1}
	\hat{p}
	\ifx\tempa\tempb\else
		_{{\scriptscriptstyle#1}}
	\fi
}   
\newcommand{\rr}[2]{r_{#1}}   
\newcommand{\D}{\hat{\Xi}}  
\newcommand{\Dt}{\hat{\Xi}_{\scriptscriptstyle{\mathrm{T}}}}  
\newcommand{\Do}{\hat{\Xi}_{\scriptscriptstyle{\mathrm{C}}}} 
\newcommand{\Vest}{\hat{V}} 
\newcommand{\xest}{\hat{x}} 
\newcommand{\Verm}{\Vest^{\scriptscriptstyle{\mathrm{SAA}}}}
\newcommand{\Vmax}{\bar{V}}
\newcommand{\Vmx}{\Vmax_{\!\ss}}
\newcommandx{\maxelem}[4][1=i,2=\nModes,4={}]{\max \{#3_#1 #4\}_{#1 \in [#2]}}
\newcommand{\err}{\varepsilon}   
\newcommand{\xx}{\bar{x}}

\providecommand{\ss}{m}
\renewcommand{\ss}{m}           
\newcommand{\rl}[1]{#1} 
\providecommand{\conf}{\beta}
\newcommand\ambca[2]{\amb_{#1}(#2)}  
\newcommand{\mbd}{\alpha}     
\newcommand{\mbdca}[2]{\mbd_{#1}(#2)} 
\newcommand{\vi}{\eta}
\newcommand{\vmx}{\bar{\eta}}

\DeclarePairedDelimiter{\inprod}{\langle}{\rangle}
\newcommand{\mx}[1]{#1_{\mathrm{max}}}
\newcommand{\mn}[1]{#1_{\mathrm{min}}}

\newcommand{\training}{training}
\newcommand{\calibration}{calibration}
\newcommand{\Div}{\mathcal{D}}
\newcommand{\TV}{\mathrm{TV}}
\newcommand{\KL}{\mathrm{KL}}
\newcommand{\Wa}{\mathrm{W}}
\newcommand{\rW}{r^{\scriptscriptstyle \mathrm{W}}}
\newcommand{\ellb}{\overline{\ell}}
\setlist[description]{leftmargin=0pt,labelindent=0cm}
\newcommand{\rev}[2]{\ifArxiv#2\else\bgroup\color{ForestGreen}#2\egroup\fi}

\Detailstrue
\Extendedtrue
\Arxivtrue
\Showillustrationtrue

\ifArxiv
\Hoeffdingtrue
\else 
\Hoeffdingfalse 
\fi

\begin{document}

\maketitle
\thispagestyle{empty}
\pagestyle{empty}

\begin{abstract}
\rev{}{
We present a novel framework
for distributionally robust optimization (DRO), 
called cost-aware DRO (\cadro).
The key idea of \cadro{} is to exploit the cost structure 
in the design of the ambiguity set to reduce conservatism.
Particularly, the set specifically constrains the
worst-case distribution
along the direction in which the 
expected cost of an approximate solution
increases most rapidly.
We prove that \cadro{} provides both a high-confidence
upper bound and a consistent estimator of the out-of-sample expected cost, and 
show empirically that it produces solutions that are 
substantially less conservative than existing DRO methods, 
while providing the same guarantees.
}
\end{abstract}

\section{Introduction}
We consider the stochastic programming problem 
\begin{equation}\label{eq:sprog-1}
   \minimize_{x \in X} \E [ \ell(x, \xi) ]
\end{equation}
with $X \subseteq \Re^n$ a nonempty, closed set of feasible decision variables,
$\xi \in \Xi$ a random variable following probability measure $\prob$, 
and $\ell: \Re^n \times \Xi \to \Re$ a known cost function.
This problem is foundational in many fields, including
operations research \cite{shapiro_LecturesStochasticProgramming_2021}, 
machine learning \cite{hastie_ElementsStatisticalLearning_2009},
and control (e.g., stochastic model predictive control) \cite{mesbah_StochasticModelPredictive_2016}.

Provided that the underlying probability measure $\prob$ is known exactly,
this problem can effectively be solved using traditional 
stochastic optimization methods \cite{royset_OptimizationPrimer_2021,shapiro_LecturesStochasticProgramming_2021}.
In reality, however, only a data-driven estimate $\hat{\prob}$ of $\prob$ is 
typically available, which may be subject to misestimations---known as \textit{ambiguity}. 
Perhaps the most obvious method for handling this issue is to disregard this 
ambiguity and instead apply a \ac{SAA} (also known as \ac{ERM} in the machine learning literature), where 
\eqref{eq:sprog-1} is solved using $\hat{\prob}$ as a \textit{plug-in} 
replacement for $\prob$.
However, this is known to produce overly optimistic estimates of 
the optimal cost 
\cite[Prop. 8.1]{royset_OptimizationPrimer_2021},
potentially resulting in unexpectedly high realizations of the
cost when deploying the obtained optimizers on new, unseen samples.
This downward bias of \ac{SAA} is closely related to the issue of overfitting, 
and commonly refered to as the \textit{optimizer's curse}
\cite{smith_OptimizerCurseSkepticism_2006a,vanparys_DataDecisionsDistributionally_2021}.

Several methods have been devised over the years to combat this undesirable behavior.
Classical techniques such as regularization and cross-validation are 
commonly used in machine learning \cite{hastie_ElementsStatisticalLearning_2009},
although typically, they are used as heuristics, providing few rigorous guarantees, 
in particular for small sample sizes.
Alternatively, the suboptimality gap of the \ac{SAA} solution may be statistically estimated 
by reserving a fraction of the dataset for independent replications \cite{bayraksan_AssessingSolutionQuality_2006}. 
However, these results are typically based on asymptotic arguments,
and are therefore not 
valid in the low-sample regime. Furthermore, although this type of approach 
may be used to \textit{validate} the \ac{SAA} solution, 
it does not attempt to \textit{improve} it, by taking into account 
possible estimation errors. 
More recently, \ac{DRO} has garnered considerable attention, 
as it provides a principled way of obtaining a 
high-confidence upper bound on the true out-of-sample cost \cite{
delage_DistributionallyRobustOptimization_2010,
mohajerinesfahani_DatadrivenDistributionallyRobust_2018,
vanparys_DataDecisionsDistributionally_2021}.
\rev{}{In particular, its capabilities to provide 
rigorous performance and safety guarantees has 
made it an attractive 
technique for data-driven and learning-based 
control \cite{
hakobyan_DistributionallyRobustRisk_2022,
schuurmans_SafeLearningbasedMPC_2023,
tac2023}}.
\ac{DRO} refers to a broad class of methods in which 
a variant of \eqref{eq:sprog-1} is solved 
where $\prob$ is replaced with a worst-case distribution within
a statistically estimated set of distributions, called an \textit{ambiguity set}.

As the theory essentially requires only that the ambiguity set 
contains the true distribution with a prescribed level of confidence, 
a substantial amount of freedom is left in the design of the geometry of 
these sets. As a result, many 
different classes of ambiguity sets have been 
proposed in the literature, e.g., 
Wasserstein ambiguity sets \cite{mohajerinesfahani_DatadrivenDistributionallyRobust_2018}, 
divergence-based ambiguity sets \cite{tac2023,
vanparys_DataDecisionsDistributionally_2021,
bayraksan_DataDrivenStochasticProgramming_2015}
and moment-based ambiguity sets \cite{
delage_DistributionallyRobustOptimization_2010,
coppens_DatadrivenDistributionallyRobust_2020
};
See \cite{rahimian_FrameworksResultsDistributionally_2022,lin_DistributionallyRobustOptimization_2022}
for recent surveys. 

Despite the large variety of existing classes of ambiguity sets, 
a common characteristic is that their design is
considered separately from the optimization problem in question.
Although this simplifies the analysis in some cases,
it may also induce a significant level of conservatism; 
In reality, we are only interested in excluding distributions from the 
ambiguity set which actively contribute to increasing the worst-case cost.
Requiring that the true distribution deviates little from the data-driven estimate 
in \textit{all directions} may therefore be unnecessarily restrictive.
This intuition motivates the introduction of a new \ac{DRO} methodology, which is aimed 
at designing the geometry of the ambiguity sets with the original problem \eqref{eq:sprog-1}
in mind. The main idea is that by only excluding those distributions that maximally affect 
the worst-case cost, higher levels of confidence can be attained without 
introducing additional conservatism to the cost estimate.

\paragraph*{Contributions}
\rev{We highlight the following contributions of this work.}{}
\begin{inlinelist}
    \item We propose a novel class of ambiguity sets 
        for\rev{ the purpose of}{} \ac{DRO}, 
   taking into account the structure of the underlying optimization problem; 
   \item We prove that the \ac{DRO} cost is both a high-confidence upper bound
   and a consistent estimate of the optimal cost of the original stochastic program \eqref{eq:sprog-1};  
   \item We demonstrate empirically that the provided ambiguity set outperforms existing alternatives. 
\end{inlinelist}

\paragraph*{Notation}

We denote
\(
    [n] = \{1, \dots, n\},
\)
for $n \in \N$.
$|S|$ denotes the cardinality of a (finite) set $S$. 
$\e_i \in \Re^n$ is the $i$th standard basis vector in 
$\Re^n$. Its dimension $n$ will be clear from context.
\ifArxiv
We denote the level sets of a function $f: \Re^n \to \Re$ as 
\( \lev_{\leq \alpha} f \dfn \{x \in \Re^{n} \mid f(x) \leq \alpha \} \). 
\fi
We write `$\as$' to signify that a random event occurs \textit{almost surely}, i.e., with probability 1.
\ifHoeffding
We denote the largest and smallest entries of a vector $v \in \Re^n$ 
as $\mx{v} \dfn \max_{i \in [n]} v_i$ and $\mn{v} = \min_{i \in [n]}v_i$, 
respectively, and define its \textit{range} as
\(
    \rg(v) \dfn \mx{v} - \mn{v}.
\)
\fi
$\delta_{X}$ is the indicator of a set $X$: $\delta_X(x) = 0$ if $x \in X$,
$+\infty$ otherwise.

\section{Problem Statement} \label{sec:problem-statement}
We will assume that the random variable $\xi$ is finitely supported,
so that
without loss of generality, we may write 
 $\Xi = \{\rl{1}, \dots, \rl{\nModes}\}$. 
This allows us to define the probability mass vector 
\(
    p = (\prob[\xi = \rl{i}])_{i=1}^{\nModes},
\)
and enumerate the \textit{cost realizations} \(\ell_i = \ell(\argdot, \rl{i})$, $i \in [\nModes]\). 
Furthermore, it will be convenient to introduce the mapping 
\(
    L: \Re^{n} \to \Re^{\nModes}
\) 
as
\( 
    L(x) = (\ell_1(x), \dots, \ell_{\nModes}(x)).
\)
We will pose the following (mostly standard) regularity assumption on the cost function.
\begin{assumption}[Problem regularity] \label{asm:regularity}
    For all $i \in [\nModes]$
    \begin{conditions}
        \item \label{asm:lsc} $\ell_i$ is continuous on $X$; 
        \item \label{asm:level-bounded} $\ellb_i \dfn \ell_i + \delta_{X}$ is level-bounded; 
    \end{conditions}
\end{assumption}
Since any continuous function is \ac{lsc}, 
\cref{asm:regularity} combined with the closedness of $X$ implies 
\textit{inf-compactness}, which ensures attainment of the minimum \cite[Thm. 1.9]{rockafellar_VariationalAnalysis_1998}.
Continuity of $\ell_i$ is used mainly 
\ifArxiv
in \cref{lem:parametric-stability}
\else\fi
to establish continuity of the 
solution mapping $\Vop$---defined below, see \eqref{eq:parametric}. 
However, a similar result can be obtained by replacing \cref{asm:lsc}
by \textit{lower semicontinuity} and \textit{uniform level-boundedness} on $X$. 
However, for ease of exposition, we will not cover this modification explicitly. 

Let $\p \in \simplex_{\nModes} \dfn \{ p \in \Re^\nModes_+ \mid \tsum_{i=1}^{\nModes} p_i = 1 \}$
denote the true-but-unknown probability mass vector, and 
define 
\( 
    V: \Re^n \times \simplex_{\nModes} \to \Re: (x, p) \mapsto \inprod{p, L(x)}
\), 
to obtain the parametric optimization problem with optimal cost and solution set 
\begin{equation} \label{eq:parametric}
    \Vop(p) = \min_{x \in X} V(x, p) \text{ and } \Xop(p) = \argmin_{x \in X} V(x, p).
\end{equation}
The solution of \eqref{eq:sprog-1} is retrieved by \rev{selecting $p = \p$.}
{solving \eqref{eq:parametric} with $p = \p$.}

Assume we have access to
a dataset $\D \dfn \{\xi_1, \dots, \xi_{\ss} \} \in \Xi^\ss$
collected i.i.d. from $\p$.
In order to avoid the aforementioned
downward bias of \ac{SAA},
our goal is to obtain a 
data-driven decision $\xest_{\ss}$ along with an estimate  
\( \Vest_{\ss} \) such that
\begin{equation}\label{eq:guarantee-coverage}
    \prob[ V(\xest_{\ss}, \p) \leq \Vest_{\ss} ] \geq 1 - \conf, 
\end{equation}
where $\conf \in (0,1)$ is a user-specified confidence level.

We address this problem by means of \acl{DRO}, 
where instead of \eqref{eq:parametric}, one solves the surrogate problem 
\begin{equation} \tag{DRO} \label{eq:dro-problem}
    \Vest_{\ss} = \min_{x \in X} \max_{p \in \amb_\ss} V(x, p). 
\end{equation}
Here, $\amb_\ss \subseteq \simplex_{\nModes}$ is a (typically data-dependent, and thus, random) set of 
probability distributions that is designed to contain the true distribution $\p$ with 
probability $1 - \conf$, ensuring that \eqref{eq:guarantee-coverage} holds. 
Trivially, \eqref{eq:guarantee-coverage} is satisfied with $\conf=0$ by taking $\amb_\ss \equiv \simplex_{\nModes}$.
This recovers a robust optimization method, i.e., $\min_{x \in X} \max_{i \in [\nModes]} \ell_i(x)$. 
Although it satisfies \eqref{eq:guarantee-coverage}, this robust approach tends to be overly conservative
as it neglects all available statistical data.
The aim of distributionally robust optimization is to additionally ensure 
that $\Vest_\ss$ is a consistent estimator, i.e., 
\begin{equation} \label{eq:guarantee-consistency}
    \lim_{m \to \infty} \Vest_m = \Vop(p^\star), \quad \as. 
\end{equation}
We will say that a class of ambiguity sets is \textit{admissible} if the 
solution $\Vest_\ss$ of the resulting \ac{DRO} problem \eqref{eq:dro-problem}
satisfies \eqref{eq:guarantee-coverage} and \eqref{eq:guarantee-consistency}.
\rev{The objective of this work}{}
Our objective
is to develop a methodology for 
constructing admissible ambiguity sets that take into account the 
structure of \eqref{eq:dro-problem} and in doing so, provide 
tighter estimates of the cost, while maintaining \eqref{eq:guarantee-coverage} with 
a given confidence level $\conf$.

\section{Cost-Aware DRO} \label{sec:cadro}
In this section, we describe the proposed \ac{DRO}
framework, which we will refer to as \textit{cost-aware \ac{DRO}} (\cadro).
The overall method is summarized in \cref{alg:cadro}. 

\subsection{Motivation} \label{sec:motivation}

 \begin{figure}[ht!]
    \centering
    \begin{minipage}{0.35\linewidth}
     \centering
     \includegraphics[height=\textwidth]{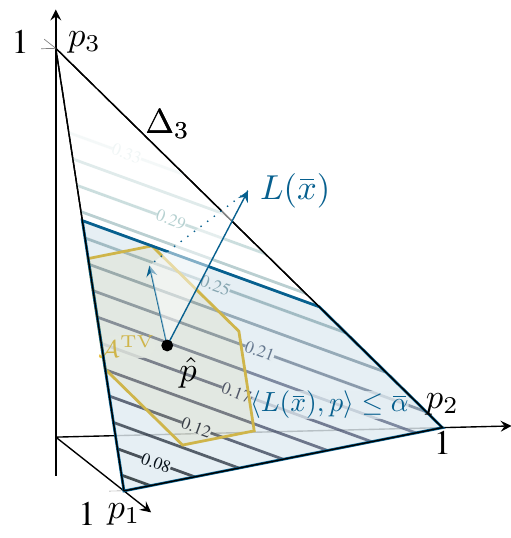}
    \end{minipage}
    \begin{minipage}{0.35\linewidth}
     \centering
     \includegraphics[height=\textwidth]{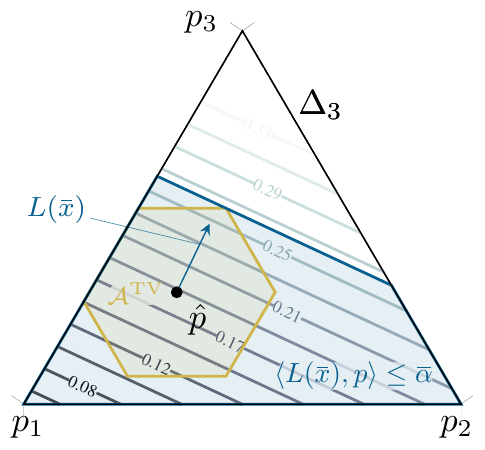}
    \end{minipage}
    \caption{
     Conceptual motivation for the structure of the
     ambiguity set \eqref{eq:ambiguity-shape}. 
     The cost contour lines 
     $\{ p \in \simplex_{3}\mid \inprod{L(\xx), p} = \alpha \}$ corresponding to
     some $\xx \in X$ are shown for increasing values of $\alpha$ (dark to light), 
     together with the sets $\ambTV \dfn \simplex_{3} \cap \ball_1(\ph, \varrho)$
     and $\amb \dfn \{p \in \simplex_{3} \mid \inprod{L(\xx), p} \leq \bar{\alpha} \}$.
     Here, $\varrho > 0$ is determined to satisfy \eqref{eq:inclusion}
     and $\bar{\alpha} = \max_{p \in \ambTV} \inprod{L(\xx), p}$.
     Since $\ambTV \subset \amb$, $\amb$ satisfies \eqref{eq:inclusion}
     with a higher confidence level $1-\conf$, but nevertheless, we have
     $\max_{p \in \amb} V(\xx, p) = \max_{p \in \ambTV} V(\xx, p)$.
    }
    \label{fig:rationale}
 \end{figure}

We start by providing some intuitive motivation.
Consider the problem \eqref{eq:dro-problem}.
In order to provide a guarantee of the form
\eqref{eq:guarantee-coverage}, 
it obviously suffices to design $\amb_\ss$ 
such that 
\begin{equation}  \label{eq:inclusion}
    \prob[ \p \in \amb_\ss ] \geq 1 - \conf. 
\end{equation}
However, this condition alone still leaves 
a considerable amount of freedom to the designer. 
A common approach is to select 
$\amb_{\ss}$ to be a ball (expressed in some 
statistical metric/divergence) 
 around an empirical 
estimate $\ph$ of the distribution.
Depending on the choice of metric/divergence 
(e.g., total variation \cite{rahimian_IdentifyingEffectiveScenarios_2019},
 Kullback-Leibler \cite{vanparys_DataDecisionsDistributionally_2021},
  Wasserstein \cite{mohajerinesfahani_DatadrivenDistributionallyRobust_2018}, \dots
), 
several possible variants may be obtained. 
Using concentration inequalities, 
one can then select the appropriate radius of this ball,
such that \eqref{eq:inclusion} is satisfied.
A drawback of this approach, however, 
is that the construction of $\amb_\ss$ is 
decoupled from the original problem \eqref{eq:sprog-1}.
Indeed, given that $\amb_{\ss}$ takes the form of a ball,
\eqref{eq:inclusion} essentially requires the deviation
of $\ph$ from $\p$ to be small \textit{along every direction}.
If one could instead enlarge the ambiguity set
without increasing the worst-case cost, then \eqref{eq:inclusion} 
could be guaranteed for smaller values of $\conf$ without 
introducing additional conservatism.
This idea is illustrated in \cref{fig:rationale}.

Conversely, for a fixed confidence level $\conf$, 
one could thus construct a smaller upper bound $\Vest_{\ss}$, 
by restricting the choice of $p$ only in a judiciously selected direction.
Particularly, we may set
$
\amb_{\ss} 
= 
\{p \in \simplex_{\nModes} \mid \inprod{L(\xx), p} \leq \mbd_{\ss} \}
$
for some candidate solution $\xx \in X$,
where $\mbd_{\ss}$ is the smallest (potentially data-dependent) quantity 
satisfying \eqref{eq:inclusion}.
This directly yields an upper bound on the estimate
$\Vest_\ss$. Namely, 
for $\xop \in \Xop(\p)$, 
we have with probability $1 - \conf$,
\[ 
\begin{aligned}
  V(\xop, \p)
    &\stackrel{\hypertarget{a}{(a)}}{\leq} 
  V(\xest_{\ss}, \p)
    \leq \max_{p \in \amb_{\ss}} V(\xest_{\ss}, p) = \Vest_\ss \\
    &= \min_{x \in X} \max_{p \in \amb_\ss} V(x,p) 
    \stackrel{\hypertarget{b}{(b)}}{\leq} \max_{p \in \amb_{\ss}} V(\bar{x}, p) = \mbd_{\ss}.
\end{aligned}
\]
Here, inequalities \hyperlink{a}{(a)} and \hyperlink{b}{(b)}
become equalities when $\xest_m = \xop = \xx$.
Thus, a reasonable aim would be to select $\xx$ 
to be a good approximation of $\xop$.
We will return to the matter of selecting $\xx$ in \cref{sec:selection-v}. 
First, however, 
we will assume $\xx$ to be given and focus on 
establishing the coverage condition \eqref{eq:inclusion}.

\subsection{Ambiguity set parameterization and coverage}\label{sec:coverage}

Motivated by the previous discussion,
we propose a family of ambiguity sets parameterized 
as follows. 
Let $v \in \Re^{\nModes}$ be a fixed vector
(we will discuss the choice of $v$ in \cref{sec:selection-v}).
Given a sample $\D = \{ \xi_{1}, \dots, \xi_{\ss} \}$ of size $|\D| = \ss$
drawn i.i.d. from $\p$, 
we consider ambiguity sets of the form
\rev{
let 
\begin{equation} \label{eq:empirical}
    \ph[\D] \dfn \tfrac{1}{\ss} \tsum_{k=1}^{\ss} \e_{\xi_k}
\end{equation}
denote the corresponding empirical probability mass vector.
Furthermore, let $v \in \Re^{\nModes}$ be a fixed vector
(we will discuss the choice of $v$ in \cref{sec:selection-v}).
We then consider ambiguity sets of the form
}{}
\begin{equation} \label{eq:ambiguity-shape}
\rev{
\ambca{\D}{v} \dfn \{ p \in \simplex_{\nModes} \mid \inprod{p - \ph[\D],v} 
\leq \rr{|\D|}{\conf} \rg(v) \}, 
}{
\ambca{\D}{v} \dfn 
\{ p \in \simplex_{\nModes} \mid \inprod{p, v} \leq \mbdca{\D}{v} \}, 
}
\end{equation}
\rev{
where $r_{|\D|}$ is 
selected to ensure that \eqref{eq:inclusion} holds
for $\amb_\ss = \ambca{\D}{v}$, 
as shown in the following result.
Note that the probability in \eqref{eq:inclusion} is taken 
with respect to the dataset $\D$.
\begin{proposition}[Coverage] \label{prop:coverage}
    Let $v \in \Re^{\nModes}$ be fixed and let  
    $\ambca{\D}{v}$ be defined as in \eqref{eq:ambiguity-shape}, 
    with respect to a random sample $\D$ with $|\D| = \ss$, drawn i.i.d. from $\p \in \simplex_\nModes$. 
    If 
    \begin{equation}\label{eq:radius}
        \rr{\ss}{\conf} = \min\big\{1, \sqrt{\tfrac{\log \nicefrac{1}{\conf}}{2\ss}}\big\}, 
    \end{equation}
    then $\prob[\p \in \ambca{\D}{v}] \geq 1 - \conf$. 
\end{proposition}
\begin{proof}
    Define 
    \(
        y_k \dfn \langle p - \e_{\xi_k}, v \rangle, 
    \)
    so that \( \tfrac{1}{\ss} \tsum_{k=1}^{\ss}y_k = \inprod{p - \ph_\ss, v} \). 
    Since $v$ is fixed, $y_k, k \in [\ss]$ are i.i.d., and we have 
    $\E[y_k] = 0$ and 
    \ifArxiv
    (by \cref{lem:aux-upper-bound}),  
    \fi
    $|y_k| \leq \rg(v)$, $\forall k \in \N$.
    This establishes the (vacuous) case $\rr{\ss}{\conf} = 1$ in \eqref{eq:radius}. 
    For the nontrivial case,
    we apply
    Hoeffding's inequality
    \cite[eq. 2.11]{wainwright_HighdimensionalStatisticsNonasymptotic_2019} 
    \begin{equation} \label{eq:confidence-proof}
    \prob \big[
        \tfrac{1}{\ss} \tsum_{k=1}^{\ss} y_k > t 
    \big] \leq \exp\left( \tfrac{-2 m t^2}{\rg(v)^2} \right).   
    \end{equation}
    Setting $t = \rr{\ss}{\conf} \rg(v)$,
    equating the right-hand side of \eqref{eq:confidence-proof} to the desired confidence level $\conf$,
    and solving for $\rr{\ss}{\conf}$ yields the desired result.
\end{proof}
}{
    where $\mbd: \Xi^{\ss} \times \Re^{\nModes} \ni (\D, v) \mapsto \mbdca{\D}{v}
    \in \Re$ 
is a data-driven estimator for $\inprod{\p, v}$, 
selected to satisfy the following assumption, 
which implies that \eqref{eq:inclusion} holds 
for $\amb_\ss = \ambca{\D}{v}$.
\begin{assumption} \label{asm:mean-bound} 
\( 
\prob[ \inprod{\p, v} \leq \mbdca{\D}{v}  ] \geq 1 - \conf, 
    \; \forall v \in \Re^{\nModes}. 
\)
\end{assumption}
Note that the task of selecting $\mbd$ to satisfy \cref{asm:mean-bound}
is equivalent to finding 
a high-confidence upper bound on the mean of the scalar random variable
$\inprod{v, \e_{\xi}}$, $\xi \sim \p$.
It is straightforward to derive 
such bounds by bounding the 
deviation of a random variable from its empirical mean
using classical concentration inequalities like Hoeffding's inequality
\ifHoeffding
. 
\begin{proposition}[Hoeffding bound] \label{prop:hoeffding}
    Fix $v \in \Re^{\nModes}$ and let  
    $\D$ with $|\D| = \ss$, be an i.i.d. sample from $\p \in \simplex_\nModes$,
    with empirical distribution 
    \(\ph[\D] = \tfrac{1}{\ss} \tsum_{\xi \in \D} \e_{\xi}\).
    Consider the bound
    \begin{equation} \label{eq:hoeffding-bound} 
        \mbdca{\D}{v} = \inprod{v, \ph[\D]} + \rr{\ss}{\conf} \rg(v). 
    \end{equation} 
    This bound satisfies \cref{asm:mean-bound}, if $\rr{\ss}{\conf}$ satisfies
    \begin{equation}\label{eq:radius}
        \rr{\ss}{\conf} = \min\big\{1, \sqrt{\tfrac{\log (\nicefrac{1}{\conf})}{2\ss}}\big\}. 
    \end{equation}
\end{proposition}
\begin{proof}
    Define 
    \(
        y_k \dfn \langle \p - \e_{\xi_k}, v \rangle, 
    \)
    so that \( \tfrac{1}{\ss} \tsum_{k=1}^{\ss}y_k = \inprod{\p - \ph_\ss, v} \). 
    Since $v$ is fixed, $y_k, k \in [\ss]$ are i.i.d., and we have 
    $\E[y_k] = 0$ and 
    \ifArxiv
    (by \cref{lem:aux-upper-bound}),  
    \fi
    $|y_k| \leq \rg(v)$, $\forall k \in \N$.
    This establishes the (vacuous) case $\rr{\ss}{\conf} = 1$ in \eqref{eq:radius}. 
    For the nontrivial case,
    we apply
    Hoeffding's inequality
    \cite[eq. 2.11]{wainwright_HighdimensionalStatisticsNonasymptotic_2019} 
    \begin{equation} \label{eq:confidence-proof}
    \prob \big[
        \tfrac{1}{\ss} \tsum_{k=1}^{\ss} y_k > t 
    \big] \leq \exp\left( \tfrac{-2 m t^2}{\rg(v)^2} \right).   
    \end{equation}
    Setting $t = \rr{\ss}{\conf} \rg(v)$,
    equating the right-hand side of \eqref{eq:confidence-proof} to the desired confidence level $\conf$,
    and solving for $\rr{\ss}{\conf}$ yields the desired result.
\end{proof}
\else
\cite[Prop. III.2]{cdc2023_arxiv}.
\fi 
Although attractive for its simplicity, this type of bounds has the 
drawback that it applies a constant offset (depending only on the 
sample \textit{size}, not the data) to the empirical mean, 
which may be conservative, especially for small samples.
Considerably sharper bounds 
can be obtained through a more direct approach. 
In particular, we will focus our attention on the following result
due to Anderson \cite{anderson_ConfidenceLimitsExpected_1969}, 
which is a special case of the framework presented in 
\cite{coppens_RobustifiedEmpiricalRisk_2023}. 
\ifHoeffding
\ifArxiv
We provide an experimental comparison between the bounds in 
Appendix~\ref{sec:comparison-hoeffding}.
\fi
\fi 

\begin{proposition}[{Ordered mean bound \cite{coppens_RobustifiedEmpiricalRisk_2023}}]
    \label{prop:orm}
	Let $\vi_{k} \dfn \inprod{v, \e_{\xi_k}}$, $k \in [\ss]$, 
    so that $\E[\vi_{k}] = \inprod{v, \p}$.
	Let \(\ord{\vi}{1}
		\leq
		\ord{\vi}{2}
        \leq \dots 
        \leq \ord{\vi}{\ss}
        \leq \vmx\) denote the sorted sequence, with ties broken arbitrarily, 
    where $\vmx \dfn \max_{i \in [\nModes]} v_i$.
	Then, there exists a $\gamma \in (0,1)$ such that 
    \cref{asm:mean-bound} holds for
    \begin{equation} \label{eq:anderson}
		\mbdca{\D}{v} =
		\big(
            \tfrac{\kthr}{\ss} - \gamma
		\big)
        \ord{\vi}{\kthr}
    + \tsum_{i = \kthr+1}^{\ss} \tfrac{\ord{\vi}{i}}{\ss} + \gamma \vmx,
    \; \kthr = \ceil{\ss \gamma}.
	\end{equation}
\end{proposition}
For finite $\ss$, the smallest value of $\gamma$ 
ensuring that \cref{prop:orm} holds, can be computed 
efficiently by solving a scalar root-finding problem 
\cite[Rem. IV 3]{coppens_RobustifiedEmpiricalRisk_2023}. 
Furthermore, it can be shown that the result holds for
\cite[Thm. 11.6.2]{wilks_MathematicalStatistics_1963}
\begin{equation} \label{eq:gam-asymptotic}  
    \gamma = \sqrt{\tfrac{\log(\nicefrac{1}{\conf})}{2 \ss}},
    \text{ for sufficiently large }
    \ss.
\end{equation}
This asymptotic expression will be useful when establishing theoretical guarantees
in \Cref{sec:theory}.
}
\subsection{Selection of \texorpdfstring{$v$}{v}} \label{sec:selection-v}

The proposed ambiguity set \eqref{eq:ambiguity-shape}
depends on a vector $v$.
As discussed in \cref{sec:motivation}, 
we would ideally take $v = L(\xop)$ with $\xop \in \Xop(\p)$. 
However, since this ideal is obviously out of reach,
we instead look for suitable approximations. 
In particular, we propose to use 
the available dataset $\D$ 
in part to select $v$ to approximate $L(\xop)$, and in part to 
 calibrate the \rev{estimate $\p$ and calibrate the ambiguity set parameter 
\eqref{eq:radius}}{mean bound $\mbd$}.

To this end, we will partition the available dataset $\D$ 
into a \textit{\training} set and a \textit{\calibration} set. 
Let $\trs: \N \to \N$ 
be a user-specified function determining the size of the \emph{\training} set, 
which satisfies
\begin{subequations} \label{eq:def-tau}
    \begin{align}
        \label{eq:tau-upper-bound}
        \trs(\ss) &\leq c \ss \; \text{ for some } c \in (0,1); \text{ and }\\ 
        \label{eq:tau-limit}
        \trs(\ss) &\to \infty \; \text{ as } \ss \to \infty. 
    \end{align}
\end{subequations}
Correspondingly, let $\{ \Dt, \Do \}$ be a partition of 
$\D$, i.e., $\Dt \cap \Do = \emptyset$ and $\Dt {}\cup{} \Do = \D$. 
Given that $|\D| = \ss$, we ensure that 
$|\Dt| = \trs(\ss)$ and thus $|\Do| = \ss' \dfn \ss - \trs(\ss)$.
Note that by construction, $\ss' \geq (1-c) \ss$, with $c \in (0,1)$, 
and thus, both $|\Dt| \to \infty$ and $|\Do| \to \infty$ as $\ss \to \infty$. 
Due to the statistical independence of the elements in $\D$, 
it is inconsequential how exactly the individual data points
are divided into $\Dt$ and $\Do$. 
Therefore, without loss of generality, we may take
$\Dt = \{\xi_1, \dots, \xi_{\trs(\ss)}\}$ and
$\Do = \{\xi_{\trs(\ss) + 1}, \dots, \xi_{\ss}\}$.

With an independent dataset $\Dt$ at our disposal, 
we may use it to design a mapping $v_{\trs(\ss)}: \Xi^{\trs(\ss)} \to \Re^{\nModes}$, 
whose output will be a 
data-driven estimate of $L(\xop)$. For ease of notation, we will omit the explicit dependence 
on the data, i.e., we write $v_{\trs(\ss)}$ instead of $v_{\trs(\ss)}(\Dt)$.
We propose the following construction. 
Let $\ph_{\trs(\ss)} = \tfrac{1}{\trs(\ss)} \sum_{k=1}^{\trs(\ss)} \e_{\xi_{k}}$
denote the empirical distribution of $\Dt$ and set 
\begin{equation} \label{eq:primer}
    \begin{aligned}
      v_{\trs(\ss)} &= L(\xx_{\trs(\ss)}), \text{ with }\\
      \xx_{\trs(\ss)} &\in \argmin_{x \in X} V(x, \ph_{\trs(\ss)}). 
    \end{aligned}
\end{equation}

\begin{remark}
    We underline that although \eqref{eq:primer}
    is a natural choice, several alternatives
    for the \textit{\training} vector
    could in principle be considered.
    To guide this choice, \cref{lem:consistency-conditions}
    provides sufficient conditions on \rev{}{the combination of $\mbd$ and} $v_{\trs(\ss)}$ 
    to ensure consistency of the method.
\end{remark}

Given $v_{\trs(\ss)}$ as in \eqref{eq:primer}, 
we will from hereon use the following shorthand notation whenever convenient: 
\begin{equation} \label{eq:amb-shorthand}
    \amb_{\ss} \dfn \ambca{\Do}{v_{\trs(\ss)}},
    \quad \mbd_{\ss} \dfn \mbdca{\Do}{v_{\trs(\ss)}},
\end{equation}
with $\ambca{\Do}{v_{\trs(\ss)}}$ as in \eqref{eq:ambiguity-shape}.
We correspondingly obtain the cost estimate 
$\Vest_{\ss}$ according to \eqref{eq:dro-problem}.

\subsection{Selection of \texorpdfstring{$\trs$}{τ(m)}}
Given the conditions in \eqref{eq:def-tau}, 
there is still some flexibility in the choice of $\trs(\ss)$, 
which defines a trade-off between the quality of $v_{\trs(\ss)}$ as 
an approximator of $L(\xop)$ and the size of the ambiguity set 
$\amb_\ss$.

An obvious choice is to reserve a fixed fraction of the 
available data for the \emph{\training} set, i.e., 
set $\nicefrac{\trs(\ss)}{\ss}$ equal to some constant.
However, for low sample counts $\ss$, 
\rev{
    the parameter $r_{\ss'}$
}{
    the mean bound $\mbd_{\ss}$
}
will typically be large and
thus $\amb_{\ss}$ will not be substantially smaller than the unit simplex $\simplex_{\nModes}$, regardless of $v_{\trs(\ss)}$.
As a result, the obtained solution will also be rather insensitive to $v_{\trs(\ss)}$. 
In this regime, it is therefore preferable to reduce 
\rev{
$\rr{\ss'}{\conf}$
}{
    the conservativeness of $\mbd_{\ss}$
} 
quickly by 
using small values of $\nicefrac{\trs(\ss)}{\ss}$ (i.e., large values of $\ss' = \ss - \trs(\ss)$).

Conversely, for large sample sizes,
\rev{
$\rr{\ss'}{\conf}$ will be small,
}{
    $\mbd_{\ss}$ is typically a good approximation of 
    $\inprod{\p, v_{\trs(\ss)}}$
}
and the solution to \eqref{eq:dro-problem} 
will be more strongly biased to align with $v_{\trs(\ss)}$.
Thus, the marginal benefit of improving the quality of $v_{\trs(\ss)}$ 
takes priority over reducing \rev{
$\rr{\ss'}{\conf}$
}{
$\mbd_{\ss}$
},
and large fractions $\nicefrac{\trs(\ss)}{\ss}$ become preferable.
Based on this reasoning, we propose the heuristic
\begin{equation} \label{eq:choice-tau}
    \trs(\ss) = \floor{\mu \nu \tfrac{\ss (\ss + 1)}{ \mu \ss + \nu}}, \quad \mu,\nu \in (0, 1).
\end{equation}
Note that $\mu$ and $\nu$ are the limits of $\nicefrac{\trs(\ss)}{\ss}$ as 
$\ss\to0$ and $\ss\to \infty$, respectively. Eq.~\eqref{eq:choice-tau} then interpolates between these 
extremes, depending on the total amount of data available.
We have found $\mu=0.01, \nu=0.8$ to be suitable choices for several test problems.

\subsection{Tractable reformulation}
The proposed ambiguity set takes the form of a polytope,
and thus, standard reformulations based on conic ambiguity sets 
apply directly \cite{sopasakis_RiskaverseRiskconstrainedOptimal_2019a}.
Nevertheless, as we will now show, 
a tractable reformulation of \eqref{eq:dro-problem}
specialized to the ambiguity set \eqref{eq:ambiguity-shape}
may be obtained,
which requires fewer auxiliary variables and constraints
\rev{
than the reformulation obtained from applying
the results in \cite{sopasakis_RiskaverseRiskconstrainedOptimal_2019a} directly}{}.

\begin{proposition}[Tractable reformulation of \eqref{eq:dro-problem}] \label{prop:reform-dro}
    Fix parameters
    $\ph \in \simplex$, $v \in \Re^{\nModes}$, and \rev{
    $r > 0$}{
    $\mbd \in \Re$
    }
    and 
    let 
    \(
        \amb = \{ p \in \simplex_{\nModes} \mid 
        \rev{\inprod{p - \ph, v} \leq r \rg(v)}
        {\inprod{p, v} \leq \mbd
    }\} 
    \)
    be an ambiguity set of the form \eqref{eq:ambiguity-shape}.
    Denoting $V_{\amb} \dfn \min_{x \in X} \max_{p \in \amb} V(x,p)$, we have 
    \begin{equation} \label{eq:DRO-reform}
        \begin{aligned} 
        V_{\amb}
            = 
            \min_{\substack{
                x \in X,
                \lambda \geq 0
        }} \lambda \rev{(\inprod{\ph, v} + r \rg(v))}{\mbd} + \max_{i \in [\nModes]} \{ \ell_i(x) - \lambda v_i \}. 
        \end{aligned}
    \end{equation}

\end{proposition}
\begin{proof}
    \rev{
    For an ambiguity set of the form \eqref{eq:ambiguity-shape}, 
    the inner maximization of the problem \eqref{eq:dro-problem} takes the form 
    }{Let}
    \( 
        g(z) \dfn \max_{p \in \simplex_{\nModes}} \{ \inprod{p, z} \mid \inprod{p, v} \leq \mbd\},
    \)
    where $z \in \Re^{\nModes}$ and 
    $\mbd \rev{= \inprod{\ph, v} + r \rg(v)}{}$ 
    are constants with respect to $p$. By strong duality of 
    \acl{LPing} \cite{ben-tal_LecturesModernConvex_2001}, 
    \begin{equation*} 
        \begin{aligned}
            g(z) &= \min_{\lambda \geq 0} \max_{p \in \simplex_{\nModes}} \inprod{p, z} - \lambda (\inprod{p, v} - \mbd) \\ 
                 &= \min_{\lambda \geq 0} \lambda \mbd + \max_{p \in \simplex_{\nModes}} \inprod{p, z - \lambda v} \\
        \end{aligned}
    \end{equation*}
    Noting that $\max_{p \in \simplex_{\nModes}} y = \max_{i \in [\nModes]} y_i,$ $\forall y\in \Re^\nModes$ and that $V_{\amb} = \min_{x \in X} g(L(x))$, we obtain \eqref{eq:DRO-reform}.
\end{proof}

If the functions $\{ \ell_i \}_{i \in [\nModes]}$ are convex, then \eqref{eq:DRO-reform}
is a convex optimization problem, which can be solved efficiently using 
off-the-shelf solvers.
In particular, if they are convex, piecewise affine functions,
then it reduces to \iac{LP}.
\ifDetails
For instance, introducing a scalar epigraph variable, 
one may further rewrite \eqref{eq:DRO-reform} as 
\begin{equation} \label{eq:reform-epigraph-relax}
    \begin{aligned}
        \min_{\substack{
            x \in X,
            \lambda \geq 0,
            z \in \Re
    }} \{ \lambda \rev{(\inprod{\ph, v} + r \rg(v))}{\mbd} + z  \mid L(x) - \lambda v \leq z\1 \}, 
    \end{aligned}
\end{equation}
which avoids the non-smoothness of the pointwise maximum in \eqref{eq:DRO-reform}
at the cost of a scalar auxiliary variable.
Even for general (possibly nonconvex) choices of $\ell_i$,
\eqref{eq:DRO-reform} is a standard nonlinear program,
which can be handled by existing solvers.
\fi

We conclude the section by summarizing the described steps in \cref{alg:cadro}. 
\begin{algorithm}
   {\small
   \caption{\cadro}
   \label{alg:cadro}
   \begin{algorithmic}
      \Require
       i.i.d. dataset $\D = \{\xi_1, \dots, \xi_{\ss} \}$;
        $\trs(\ss)$ (cf. \eqref{eq:def-tau}); Confidence parameter $\beta \in (0,1)$. 
      \Ensure $(\Vest_\ss, \xest_\ss)$ satisfy \eqref{eq:guarantee-coverage}--\eqref{eq:guarantee-consistency} \Comment{Cf. \cref{sec:theory}}
      \State $\Dt \gets \{\xi_1, \dots, \xi_{\trs(\ss)}\}, \;  \Do \gets \{\xi_{\trs(\ss)+1}, \dots, \xi_{\ss}\}$
      \State $v_{\trs(\ss)} \gets$ evaluate \eqref{eq:primer}
      \State
      $(\Vest_\ss, \xest_{\ss}) \gets$ solve \eqref{eq:dro-problem} with $\amb_\ss = \ambca{\Do}{v_{\trs(\ss)}}$ \Comment{Use \eqref{eq:DRO-reform}}
   \end{algorithmic}
   }
\end{algorithm}

\section{Theoretical Properties} \label{sec:theory}
We will now show that the proposed scheme possesses the
required theoretical properties, namely to provide
\begin{inlinelist*}
	\item an upper bound to the out-of-sample cost, with high probability (cf. \eqref{eq:guarantee-coverage})
	\item a consistent estimate of the true optimal cost (cf. \eqref{eq:guarantee-consistency}).
\end{inlinelist*}
\ifArxiv
	Let us start with the first guarantee, which follows almost directly by construction.
\fi

\begin{thm}[Out-of-sample guarantee] \label{thm:out-of-sample-cadro}
Fix $\ss > 0$, and 
let $\Vest_\ss$, $\xest_\ss$ be generated by \cref{alg:cadro}. Then, 
\begin{equation}
    \prob[ V(\xest_{\ss}, \p) \leq \Vest_{\ss} ] \geq 1 - \conf. 
\end{equation}
\end{thm}
\ifArxiv
\begin{proof}
    If $\p \in \amb_{\ss}$, then
    \begin{equation} \label{eq:proof-conditional-bound}
        \Vmx(x) \dfn \max_{p \in \amb_{\ss}} V(x,p) \geq V(x, \p),\; \forall x \in X . 
    \end{equation}
    Since $\xest_{\ss} \in \argmin_{x \in X} \Vmx(x)$, 
    \rev{we therefore have that}
    {\eqref{eq:proof-conditional-bound} implies that}
    $V(\xest_\ss, \p) \leq \rev{\bar{V}(\xx)}{\Vmx(\xest_\ss)} = \Vest_{\ss}$, 
    \rev{}{where the last equality holds by definition \eqref{eq:dro-problem}.}
    Consequently, 
    \rev{}{$\p \in \amb_{\ss} \implies V(\xest_\ss, \p) \leq \Vest_\ss$}, 
    and thus 
    \(
        \prob[V(\xest_\ss, \p) \leq \Vest_{\ss}] \geq \prob[\p \in \amb_{\ss}]. 
    \)
    Since $v_{\trs(\ss)}$ is constructed independently from $\Do$, 
    \rev{\cref{prop:coverage}}{\cref{asm:mean-bound} ensures that \eqref{eq:inclusion}}
    holds with respect to $\amb_{\ss}=\ambca{\Do}{v_{\trs(\ss)}}$, 
    establishing the claim. 
\end{proof}
\fi

We now turn our attention to the matter of consistency.
That is, we will show that under suitable conditions on
the \rev{}{mean bound $\mbd$ and} the \emph{\training} vector $v$ in \eqref{eq:ambiguity-shape},
$\Vest_{\ss}$ converges almost surely to the
true optimal value, as the sample size $\ss$ grows to infinity.
We will then conclude the section by demonstrating that for the
\rev{
	\emph{\training{}} vector \eqref{eq:primer},
	proposed in \cref{sec:selection-v},
}{
	choices proposed in \cref{sec:coverage,sec:selection-v},
}
the aforementioned conditions hold.

\begin{lem}[consistency conditions] \label{lem:consistency-conditions}
        Let $\Dt$, $\Do$ be two 
        independent samples from $\p$, with sizes 
        $|\Dt| = \trs(\ss)$ and $|\Do| = \ss' \dfn \ss - \trs(\ss)$.
        Let $\ph_{\ss'} \dfn \rev{
        \ph[\Do]}{
            \tfrac{1}{\ss'} \sum_{\xi \in \Do} \e_\xi
        }$ denote the empirical distribution of
        the \calibration{} set $\Do$\rev{, as in \eqref{eq:empirical}}{}.
        \rev{Let $\rr{\ss'}{\conf}$ be given by \eqref{eq:radius}.}{}
        If $v_{\trs(\ss)} = L(\xx_{\trs(\ss)})$, with
        \newcommand{\fullmbd}{\mbdca{\Do}{v_{\trs(\ss)}}}
        $\xx_{\trs(\ss)}$, \rev{}{$\mbd_{\ss} = \fullmbd$} chosen to ensure 
        \begin{conditions}
            \rev{}{\item \label{cond:lower-bound} 
            $\inprod{\ph[\ss'], v_{\trs(\ss)}} \leq \fullmbd,\;\as;$}
            \item \label{cond:limit} \rev{\(
                \inprod{ \ph_{\ss'}, v_{\trs(\ss)} } \to \Vop(\p), \quad \as; 
            \)}{
            \( 
            \limsup_{\ss \to \infty} \fullmbd \leq \Vop(\p),\; \as.
            \)
            }
            \rev{
                \item \label{cond:uniform-bound} 
                    \( \nicefrac{\nrm{v_{\trs(\ss)}}_{\infty}}{\sqrt{\ss}} 
                    \to 0, \quad \as.
                    \)
            }{}
        \end{conditions}
        Then 
        \(
            \Vest_\ss \to \Vop(\p), \; \as, 
        \)
        where $\Vest_\ss$ is given by \eqref{eq:dro-problem}.
\end{lem}
\begin{proof}
    Let $\Vmx(x) \dfn \max_{p \in \amb_\ss} \inprod{p, L(x)}$.
    \rev{
    Since $\rr{\ss}{\conf} > 0$ and $\rg(v) \geq 0$ by definition, 
    }{}
    It is clear from
    \rev{}{\cref{cond:lower-bound} and}
    \eqref{eq:ambiguity-shape} that
    $\ph_{m'} \in \amb_\ss$.
    Let us furthermore define 
    \( \err_{\ss}(x) = L(x) - L(\xx_{\trs(\ss)}) \). 
    Then, by
    \ifArxiv
        \cref{lem:upper-bound-dro-cost}
    \fi
    , we have for all $x \in X$, 
    \(
          \inprod{\ph_{\ss'}, L(x)} 
              \leq
          \Vmx(x)\rev{, \text{ and }\\               
          \Vmx(x)
            \leq
        \inprod{\ph_{\ss'}, v_{\trs(\ss)}}+ 2\rr{\ss'}{\conf} \nrm{v_{\trs(\ss)}}_{\infty}\!+\nrm{\err_\ss(x)}_{\infty}.}{
        {}\leq{} \mbd_{\ss} + \nrm{\err_{\ss}(x)}_{\infty}.
    }\) 
    Minimizing with respect to $x$ yields that for all $\ss$, 
    \begin{equation} \label{eq:proof-consistency-upper-lower}
        \Verm_{\ss'} 
            \leq 
        \Vest_\ss 
            \leq
            \rev{
        \inprod{\ph_{\ss'}, L(\xx_{\trs(\ss)})} + 2\rr{\ss'}{\conf} \nrm{v_{\trs(\ss)}}_{\infty},
    }{
        \mbd_{\ss},
    }
    \end{equation}
    where $\Verm_{\ss'} \dfn \Vop(\ph_{\ss'})$ (cf. \eqref{eq:parametric}).
    \rev{
        Due to \cref{cond:limit}, the first term in the right-hand side of
    \eqref{eq:proof-consistency-upper-lower} converges to $\Vop(\p)$. 
    As for the second term,
    \eqref{eq:def-tau} ensures that $\gamma' \ss \leq  \ss' \leq \ss$,
    for some constant $\gamma' = 1 - \gamma \in (0,1)$, and thus, 
    by \eqref{eq:radius},
    $\rr{\ss'}{\conf} \rev{}{{}\leq{}} \nicefrac{C}{\sqrt{\ss}}$, 
    with constant 
    $C \dfn \sqrt{\nicefrac{\log(\nicefrac{1}{\beta})}{2\rev{}{\gamma'}}}$, 
    for sufficiently large $\ss$. 
    Therefore,
    \(
        \rr{\ss'}{\conf} \nrm{v_{\trs(\ss)}}_{\infty} 
        \rev{=}{{}\leq{}}
        C \nicefrac{\nrm{v_{\trs(\ss)}}_{\infty}}{\sqrt{\ss}} 
        \stackrel{\ref{cond:uniform-bound}}{\to} 0,\; \as,
    \)
    and thus 
    \(
    \rev{\lim_{\ss\to\infty}}{\limsup_{\ss\to \infty}} \Vest_{\ss} \leq \Vop(\p).
    \)
    Consider finally the lower bound $\Verm_{\ss'} = \Vop(\ph_{\ss'})$.
    }{
    }
    By the law of large numbers, $\ph_{\ss'} \to \p,\,\as.$ 
    Furthermore, under \cref{asm:regularity},
    \ifArxiv
        \cref{lem:parametric-stability}
    \fi
    states that the optimal value mapping $\Vop(p)$ 
    is continuous, which implies that also 
    $\Verm_{\ss'} \to \Vop(\p),\,\as.$ 
    \rev{
        concluding the proof.
    }{
        The claim then follows directly from \cref{cond:limit}.
    }
\end{proof}
Informally, \cref{lem:consistency-conditions} requires that
\rev{
	$v_{\trs(\ss)}$ is eventually bounded---or at least does not grow faster than
	$\nicefrac{1}{\sqrt{\ss}}$,
}{
	the mean bound is bounded from below by the empirical mean,
	and from above by a consistent estimator of the optimal cost.
}
The latter excludes choices such as the robust minimizer
$\xx_{\trs(\ss)} \in \argmin \max_{i \in [\nModes]} \ell_i(x)$
in the construction of $v_{\trs(\ss)}$.
However, besides \eqref{eq:primer}, one could consider
alternatives, such as
a separate \ac{DRO} scheme to select $v_{\trs(\ss)}$.
A more extensive study of such alternatives,
however, is left for future work.
We now conclude the section by showing that \cref{eq:primer}
\ifHoeffding
\else
	\rev{}{
		and the mean bound given by \cref{prop:orm}
	}
\fi
satisfy the requirements of \cref{lem:consistency-conditions}.
\begin{thm}[Consistency -- Ordered mean bound] \label{thm:consistency}
   Let $\Vest_\ss$ be generated by \cref{alg:cadro}, for $\ss > 0$. 
   \rev{
   If $v_{\trs(\ss)}$ is chosen according to \eqref{eq:primer},
   }{
   If $\mbd_{\ss} = \mbdca{\Do}{v_{\trs(\ss)}}$ is selected
    according to \cref{prop:orm}, with $v_{\trs(\ss)}$ as in \eqref{eq:primer},
   }
   then, 
   \(
       \Vest_\ss \to \Vop(\p), \; \as. 
   \)
\end{thm}
\rev{}{
\begin{proof}
	It suffices to show that
	\cref{cond:lower-bound,cond:limit} of \cref{lem:consistency-conditions}
	are satisfied by \rev{$v_{\trs(\ss)}$}{$\mbdca{\Do}{v_{\trs(\ss)}}$}.

	\paragraph*{\Cref{cond:lower-bound}}
	Consider $\mbdca{\Do}{v}$ as in \eqref{eq:anderson} for
	an arbitrary $v \in \Re^{\nModes}$, and let
    $(\ord{\vi}{i})_{i \in [\ss']}$ denote
	$(\inprod{v, \e_{\xi}})_{\xi \in \Do}$,
	sorted in increasing order, then, we may write
    \begin{equation} \label{eq:emp-sorted}
        \inprod{\ph_{\ss'}, v} = \tfrac{1}{\ss'} \tsum_{i=1}^{\ss'} \ord{\vi}{i},
    \end{equation}
    and thus,
    \begin{equation*} 
        \begin{aligned}
            \mbdca{\Do}{v} - \inprod{\ph_{\ss'}, v}
			 & =
			\big(
			\tfrac{\kthr}{\ss'} - \gamma
			\big)
			\ord{\vi}{\kthr}
			- \tsum_{i = 1}^{\kthr} \tfrac{\ord{\vi}{i}}{\ss'} + \gamma \vmx,
			\\
			 & \labelrel{\geq}{eq:ineq-ordered}
			\big(
			\tfrac{\kthr}{\ss'} - \gamma
			\big)
			\ord{\vi}{\kthr}
			- \tfrac{\kthr}{\ss'}\ord{\vi}{k} + \gamma \vmx,
			\\
			 & = \gamma (\vmx - \ord{\vi}{k})
             \stackrel{\scriptscriptstyle(\gamma \geq 0)}{\geq} 0, 
             \quad \forall v \in \Re^{\nModes},
		\end{aligned}
	\end{equation*}
	where \eqref{eq:ineq-ordered} follows from the fact that
	$\ord{\vi}{i}$ are sorted. 
	\paragraph*{\Cref{cond:limit}}
	\ifArxiv
		By \cref{lem:bounded-v}, there
	\fi
	exists a constant $\bar{v} \geq \nrm{v_{\trs(\ss)}}_{\infty}$,
	$\forall \ss > 0,\, \as$
	.
    Therefore, using \eqref{eq:anderson} and \eqref{eq:emp-sorted}, 
	\begin{equation} \label{eq:proof-upperbound-diff}
		\begin{aligned}
        \mbd_{\ss}
        - \inprod{\ph_{\ss'}, v_{\trs(\ss)}}
			 & \leq
			(\tfrac{\kthr}{\ss'} - \gamma) \bar{v}
			+
			 \tfrac{\kthr}{\ss'} \bar{v}
            + 
            \gamma \bar{v} 
			\\
			 & =  2\bar{v}(\tfrac{\kthr}{\ss'})  
             \labelrel[1]{\leq}{rel:ceiling}
			2 \bar{v} (\gamma + \tfrac{1}{\ss'}),
		\end{aligned}
	\end{equation}
    for all $\ss' > 0$, where \eqref{rel:ceiling}
    follows from $\kthr = \ceil{\ss' \gamma}\leq \ss' \gamma + 1$.
	By construction (see \eqref{eq:def-tau} and below),
	we have that both $\trs(\ss) \to \infty$ and $\ss' \to \infty$.
	Thus, using \eqref{eq:gam-asymptotic},
	\(
	\gamma + \tfrac{1}{\ss'} =
	\sqrt{\tfrac{\log(\nicefrac{1}{\conf})}{2\ss'}}
	+
	\tfrac{1}{\ss'}
	\to 0.
	\) Combined with \eqref{eq:proof-upperbound-diff},
	this yields that
	\begin{equation} \label{eq:proof-limit}
		\limsup_{\ss \to \infty} \alpha_{\Do}(v_{\trs(\ss)}) -
		\inprod{v_{\trs(\ss)}, \ph[\Do]} \leq 0.
	\end{equation}
	Finally, by the law of large numbers,
	\(
	\ph_{\ss'} \to \p
	\) and
	\(
	\ph_{\trs(\ss)} \to \p,
	\)\as.
	Thus (under \cref{asm:regularity}),
	\ifArxiv
		\cref{cor:convergence-vp}
	\fi
	ensures that
	\(
	\lim_{\ss \to \infty} \inprod{\ph_{\ss'}, L(\xx_{\trs(\ss)})} = \Vop(\p),
	\)
	which, combined with \eqref{eq:proof-limit} yields the required result.
\end{proof}

}
\ifHoeffding
\begin{thm}[Consistency -- Hoeffding bound]
   Let $\Vest_\ss$ be generated by \cref{alg:cadro}, for $\ss > 0$. 
   If $\mbd_{\ss'} = \mbdca{\Do}{v_{\trs(\ss)}}$ is selected
    according to \cref{prop:hoeffding}, with $v_{\trs(\ss)}$ as in \eqref{eq:primer}
   then, 
   \(
       \Vest_\ss \to \Vop(\p), \; \as. 
   \)
\end{thm}
\begin{proof} 
    We show that \cref{cond:lower-bound} and \cref{cond:limit} of \cref{lem:consistency-conditions} are satisfied. 
    \paragraph*{\Cref{cond:lower-bound}} Trivial, noting that $\rr{\ss'}{\conf} > 0$
    by \eqref{eq:radius}.
    \paragraph*{\Cref{cond:limit}} By the law of large numbers, 
    we have that $\ph[\ss'] \to \p, \as$, and thus, 
    by \cref{cor:convergence-vp},
    $\inprod{ \ph_{\ss'}, v_{\trs(\ss)}} \to \Vop(\p)$.
    Furthermore, by \cref{lem:bounded-v}, there exists a constant $\bar{v}$
    such that
    $\rg(v_{\trs(\ss)}) \leq 2\bar{v}$ for all $\ss \in \N$. 
    Thus, for $\rr{\ss}{\conf}$ given by \eqref{eq:radius}, we have
    \[ 
        \begin{aligned}
            \limsup_{\ss \to \infty} \mbdca{\Do}{v_{\trs(\ss)}} 
            &\leq
            \limsup_{\ss \to \infty} \inprod{\ph_{\ss'}, v_{\trs(\ss)}} + 
            \rr{\ss}{\conf} \bar{v} \\
            &= \Vop(\p).
        \end{aligned}
    \]
    This concludes the proof.
\end{proof}
\fi

\section{Illustrative example} \label{sec:numerical}

As an illustrative example, we consider the following 
\textit{facility location problem}, adapted from \cite[Sec. 8.7.3]{boyd_ConvexOptimization_2004}. 
Consider a bicycle sharing service setting out to determine 
locations
$x^{(i)} \in X_{i} \subseteq \Re^{2}$, $i \in [n_x]$,
at which to build stalls where bikes can be taken out or returned.
We will assume that $X_{i}$ are given (polyhedral) sets,
representing areas within the city
\rev{that would be}{} suitable for constructing a new bike stall.
Let $z^{(k)} \in \Re^2$, $k \in [\nModes]$, be given points of interest 
(public buildings, tourist attractions, parks, etc.). 
Suppose that a person located in the vicinity of some point $z^{(k)}$ 
decides to rent a bike.
Depending on the availability at 
the locations $x^{(i)}$, this person may be required
to traverse a distance 
\(
    \ell_{k}(x) = \max_{ i \in [n_x] } \nrm{x^{(i)} - z^{(k)}}_2,
\)
where $x = (x^{(i)})_{i \in [n_x]}$.
With this choice of cost, \cref{eq:DRO-reform} can be cast as a second order cone program.
Thus, if the demand is distributed over $(z^{(k)})_{k \in [\nModes]}$
according to the probability mass vector $\p \in \simplex_{\nModes}$, 
then the average cost to be minimized over $X=X_{1}\times \dots \times X_{\nModes}$ is given by $V(x, \p)$ as in \eqref{eq:parametric}.
\ifShowillustration
We will solve a randomly generated instance of the problem, 
illustrated in \cref{fig:illustration-facility-location}.
\begin{figure}[ht!]
    \centering
    \includegraphics{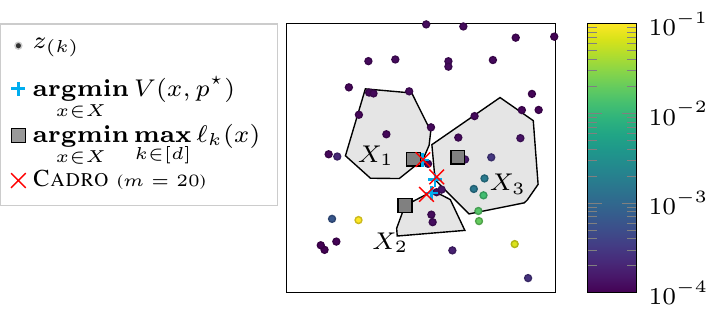}
     \vspace{-0.35cm}
    \caption{Illustration of the facility location problem. The colors of the points $z^{(k)}$ represent their probability $\p_k$.}
    \label{fig:illustration-facility-location}
\end{figure}
\fi 

As $\p$ is unknown, one has to collect data, e.g.,
by means of counting passersby at the locations $z^{(k)}$.
As this may be a costly operation, it is important
that the acquired data is used efficiently.
Furthermore, in order to ensure that the potentially large up-front
investment is justified,
we are required to provide a certificate stating that, with high confidence,
the quality of the solution will be no worse than what is predicted.
Thus, given our collected sample of size $\ss$,
our aim is to compute estimates $\Vest_\ss$,
satisfying \eqref{eq:guarantee-coverage}. 

We compare the following data-driven methods.
\begin{description}
   \item[\cadro] Solves \eqref{eq:dro-problem} according to \cref{alg:cadro}, setting $\trs(\ss)$ as in \eqref{eq:choice-tau}, with $\mu=0.01, \nu=0.8$.
   \item[$\Div$-\ac{DRO}] Solves \eqref{eq:dro-problem}, with \rev{
   a \ac{TV} ambiguity set 
   \(
        \amb_\ss = \{ p \in \simplex_{\nModes} \mid \nrm{p - \hat{p}_{\ss}}_1 \leq \rTV_{\ss}\}.
   \)
   The radius $\rTV_\ss$ is selected according to \cite[Thm 2.1]{weissman_InequalitiesL1Deviation_2003}
   \footnote{
    This is a slightly improved version of the classical 
    Bretagnolle-Huber-Carol inequality \cite[Prop. A.6.6]{vandervaart_WeakConvergenceEmpirical_2000}.
   }
    }{
    an ambiguity set of the form 
    \(
        \amb_{\ss} = \{ 
            p \in \simplex_{\nModes} \mid \Div(\ph_\ss, p) \leq r^\Div_\ss 
        \}, 
    \)    
    with $\Div \in \{\TV, \KL, \Wa\}$ the total variation,
    Kullback-Leibler, 
    and Wasserstein distance/divergence\footnote{
        We use $K_{ij} = \nrm{z^{(i)} - z^{(j)}}_2$, $i,j \in [\nModes]$ as the transportation cost.
    } (cf. \cite[Tb. I]{tac2023}).
    $\rTV_\ss, \rKL_{\ss}$ are selected according to 
    \cite[Thm 2.1]{weissman_InequalitiesL1Deviation_2003}
    \footnote{
        This is a slightly improved version of the classical 
        Bretagnolle-Huber-Carol inequality 
        \cite[Prop. A.6.6]{vandervaart_WeakConvergenceEmpirical_2000}.
    }, \cite[Thm. 5]{vanparys_DataDecisionsDistributionally_2021}, respectively, and
    $\rW_{\ss} = \max_{i,j \in [\nModes]} K_{ij} \rTV_{\ss}$ \cite{gibbs_ChoosingBoundingProbability_2002b},
    }
    ensuring that \eqref{eq:inclusion} is satisfied.
   \item[\ac{SAA}] Using the same data partition $\{\Dt, \Do\}$ as \cadro, we use 
   $\Dt$ to compute $x_\ss = \xx_{\trs(\ss)}$ as in \eqref{eq:primer}, and 
   we use $\Do$ to obtain a high-confidence upper bound 
   $\Vest_\ss = \mbdca{\Do}{L(\xx_{\trs(\ss)})}$, utilizing 
   \cref{prop:orm}.
\end{description}
\rev{}{
}
Note that \rev{\tvdro{}}{$\Div$-\ac{DRO}} does not 
require an independent data sample in order to satisfy \eqref{eq:guarantee-coverage}.

\begin{remark}
Other methods could be used to validate \ac{SAA}
(e.g., cross-validation \cite{hastie_ElementsStatisticalLearning_2009}, replications \cite{bayraksan_AssessingSolutionQuality_2006}), 
but these methods only guarantee the required confidence level asymptotically. 
In order to obtain a fair comparison, 
we instead use the same 
\rev{concentration inequality (namely \eqref{eq:radius})}{
    mean bound, namely \eqref{eq:anderson}
} for both 
\cadro{} and \ac{SAA}, so both methods provide the same theoretical guarantees.
Moreover, we note that a different data partition could be used for \ac{SAA}. 
However, preliminary experiments have indicated
that significantly increasing or decreasing $\trs(\ss)$ resulted in 
deteriorated bounds on the cost.
\end{remark}

We set $n_x = 3$,
$\nModes = 50$, $\conf = 0.01$, and apply each method
for 100 independently drawn datasets of size $\ss$. 
In \cref{fig:results}, we plot the estimated costs $\Vest_\ss$ and 
the achieved out-of-sample cost $V(\xest_\ss, \p)$, 
for increasing values of $\ss$. 
We observe that \cadro{} provides a 
sharper cost estimate $\Vest_\ss$ than the other approaches. In particular,
\rev{
\ac{SAA} performs relatively poorly for low sample sizes. In this regime, 
\cadro{} and \tvdro{} perform similarly, as initially, the ambiguity set 
will approximately cover the entire unit simplex. 
For $\ss \geq 10$, $\rr{\ss}{\beta}$ becomes non-vacuous, 
and the \cadro{} cost estimate decreases at a faster rate than that of 
\tvdro{}.
\cref{fig:results}\hyperlink{subfig:b}{(right)} additionally shows that 
the true out-of-sample cost achieved by \cadro{}
decreases more quickly than the compared methods, 
illustrating that it does not return an 
overly conservative solution which ignores the available data.
}{
the classical DRO formulations require relatively large amounts of 
data before obtaining a non-vacuous upper bound on the cost.
The right-hand panel in \cref{fig:results} shows that additionally, 
\cadro{} returns solutions which exhibit superior 
out-of-sample performance than the compared approaches, 
illustrating that it does not rely on conservative solutions to obtain 
better upper bounds.
}
\begin{figure}[ht!]
    {\footnotesize
    \centering
    \begin{minipage}{0.48\linewidth}
        \centering
        \includegraphics{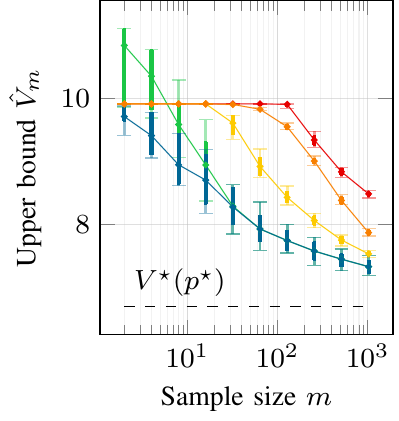}   
    \end{minipage}\hfill
    \begin{minipage}{0.48\linewidth}
        \centering
        \includegraphics{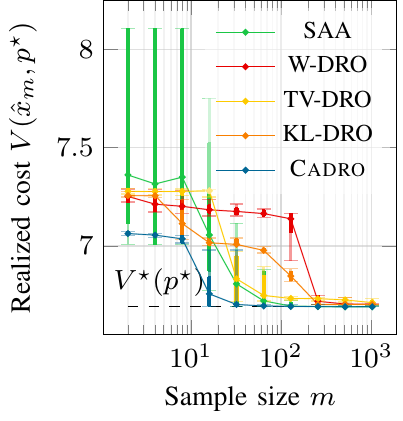}   
    \end{minipage}
    }
    \caption{
        Results of the facility location problem of \Cref{sec:numerical}. 
        (left): 
         The cost estimates $\Vest_\ss$
        satisfying \eqref{eq:guarantee-coverage} and \eqref{eq:guarantee-consistency};
        (right): 
        True out of sample cost $V(\xest_\ss, \p)$.
        The points indicate the sample mean, 
        the solid errorbars indicate the empirical 0.95 (upper and lower) quantiles
        and the semi-transparent errorbars indicate the largest and smallest values over 100 independent runs. 
    }
    \label{fig:results}
\end{figure}

\section{Conclusion and future work}
We proposed a \ac{DRO} formulation, named 
\textit{cost-aware} \ac{DRO} (\cadro), 
in which the ambiguity set is designed to only 
restrict errors in the distribution that are predicted to
have significant effects on the worst-case expected cost. 
We proved out-of-sample performance bounds and 
consistency of the resulting \ac{DRO} scheme, 
and demonstrated empirically that this approach may be used to 
robustify against poor distribution estimates at small sample sizes, 
while remaining considerably less conservative than 
existing \ac{DRO} formulations.
In future work, we aim to extend the work to continuous distributions.

\bibliographystyle{hieeetr}
\bibliography{references}

\ifArxiv
\begin{appendix}
    \subsection{Technical lemmas}
    \ifHoeffding
\begin{lem} \label{lem:aux-upper-bound}
    Let $\e_i$ denote the i'th standard basis vector.
    \[
        | \inprod{p - \e_i, v}| \leq \rg(v), 
    \]
    for all $i \in [\nModes]$, $v \in \Re^\nModes$ and $p \in \simplex_{\nModes}$.
\end{lem}
\begin{proof} For any $i \in [\nModes], v \in \Re^{\nModes}$ and $p \in \simplex_{\nModes}$,
    \[ 
    \begin{aligned}
        | \inprod{p - \e_i, v} | 
            &\leq  
        \max \{ \max_{i \in [\nModes]} \inprod{p,v} - v_i, \max_{i \in [\nModes]}  v_i-\inprod{p,v} \}\\
            &=
        \max \{ \inprod{p,v} - \mn{v} , \mx{v} - \inprod{p,v} \}\\
            &\labelrel{\leq}{rel:ineq-range}
        \mx{v} - \mn{v} = \rg(v),\\
    \end{aligned}   
    \]
    where \eqref{rel:ineq-range} follows from the fact that $\max_{p \in \simplex} \inprod{p,v} = \mx{v}$
    and 
    $\max_{p \in \simplex} -\inprod{p,v} = \max_{i \in [\nModes]} \{-v_i\} = -\mn{v}$.
\end{proof}
    \fi
\begin{lem}[Upper bound] \label{lem:upper-bound-dro-cost}
    Fix $v \in \Re^{\nModes}$ and 
    consider a sample $\D$
    \rev{with
    empirical distribution $\ph \in \simplex_{\nModes}$}{}.
    For an ambiguity set $\ambca{\D}{v}$, 
    given by \eqref{eq:ambiguity-shape} with 
    \rev{
        $\rr{|\D|}{\conf} = r$
    }{
        mean bound $\mbdca{\D}{v}$
    }, define
    \(
        \Vmax(x) \dfn \max_{p \in \ambca{\D}{v}} V(x,p).
    \)
    Then, 
    for all $x \in X$, we have
    \[ 
        \Vmax(x) \leq 
        \rev{
        \inprod{\ph, v} + 2 r \nrm{v}_\infty + \nrm{L(x) - v}_{\infty}
    }{
        \mbdca{\D}{v} + \nrm{L(x) - v}_{\infty} 
    },
        \quad \as.
    \]
\end{lem}
\begin{proof}
    Define $\err(x) = L(x) - v$ for $x \in X$.
    \rev{
    Since $\rg(v) \leq 2 \nrm{v}_\infty$,}{}
    We have 
    \[ 
    \begin{aligned}
      \Vmax (x) &= \max_{p \in \ambca{\D}{v}} \inprod{p, v} + \inprod{p, \err(x)}\\ 
                &\stackrel{\eqref{eq:ambiguity-shape}}{\leq} 
                \rev{
                    \inprod{\ph, v} + 2 r \nrm{v}_\infty}{
                    \mbdca{\D}{v}
                }
                + \max_{p \in \ambca{\D}{v}}\inprod{p, \err(x)}.
                \rev{&\leq \inprod{\ph, v} + 2 r \nrm{v}_\infty + \nrm{L(x)-v}_{\infty},}{}
    \end{aligned}
    \]
    \rev{
    where the last inequality}
    {The claim directly}
    follows because $\ambca{\D}{v} \subseteq \simplex_{\nModes}$
    \ and $\max_{p \in \simplex_{\nModes}}\inprod{p,z} = \max_{i} \{ z_i \}$, $\forall z \in \Re^{\nModes}$ \cite[Ex. 4.10]{beck_FirstOrderMethodsOptimization_2017}.
\end{proof}
\begin{lem}[Uniform level-boundedness] \label{lem:uniform-level-bounded}
    If \cref{asm:regularity}\ref{asm:level-bounded} holds, then
    \(
    V(x,p) = \inprod{p, L(x)} + \delta_{X \times \simplex_{\nModes}}(x, p)
    \)
    is level-bounded in $x$ locally uniformly in $p$.
\end{lem}
\begin{proof}
    Since $V(x,p)$ is a convex combination of $\ellb_{i}(x)$, $i \in [\nModes]$, 
    \( 
        V(x,p) \leq \alpha 
    \) implies that \( 
        \exists i \in [\nModes]: \ellb_i(x) \leq \alpha. 
    \) Therefore, 
    $\lev_{\leq \alpha} V(\argdot, p) \subseteq 
        \bigcup_{i \in [\nModes]} \lev_{\leq \alpha} \ellb_i \nfd U_\alpha$, 
        for all $p \in \simplex_{\nModes}$. 
    By \cref{asm:regularity}\ref{asm:level-bounded}, 
    $\lev_{\leq \alpha} \ellb_{i}$ is bounded for all $i \in [\nModes]$. 
    Since the union of a finite number of bounded sets is bounded,
    $U_\alpha$ is bounded. 
    Furthermore, for $p \notin \simplex_{\nModes}$, $V(x,p) = \infty$, and 
    thus $\lev_{\leq \alpha} V(\argdot,p) = \emptyset \subseteq U_\alpha,
    \forall p \notin \simplex_{\nModes}$
    Thus, 
    $\lev_{\leq \alpha} V(\argdot,p) \subseteq U_{\alpha}$
    for all $p \in \Re^{\nModes}$. 
\end{proof}
\begin{lem}[Uniform boundedness of $v_{\trs(\ss)}$] \label{lem:bounded-v}
    Let $v_{\trs(\ss)}$ be defined as in \eqref{eq:primer}. Then, 
    there exists a $\bar{v} \in \Re_+$ such that 
    \[ 
        \nrm{v_{\trs(\ss)}}_{\infty} \leq \bar{v}, \; \forall \ss \in \N, \; \as. 
    \]
\end{lem}
\begin{proof}
        By \cref{asm:regularity}, there exists
        \[
            \bar{r} \dfn \min_{x \in X} \max_{i \in [\nModes]} \ell_i(x) \geq \min_{x \in X} V(x,p), \; \forall p \in \simplex_{\nModes},
        \]
        so that, by \eqref{eq:primer},
         \begin{equation*} 
        \xx_{\trs(\ss)} \in \lev_{\leq \bar{r}} V(\argdot, \ph_{\trs(\ss)}), \; \forall \ss \in \N.
         \end{equation*}
        Since $V(x,p)$ is level-bounded uniformly in $p$
        (cf. \cref{lem:uniform-level-bounded}),
    there exists a compact set $C \subseteq \Re^{n}$ satisfying 
    \begin{equation} \label{eq:proof-x-contain-lev}
        \xx_{\trs(\ss)} \in \lev_{\leq \bar{r}} V(\argdot, \ph_{\trs(\ss)}) \subseteq C, \quad \forall \ss \in \N.
    \end{equation}
    Hence, since $\ell_i$, $i \in [\nModes]$ are continuous, they attain
    their minima $\underline{v}_{i}$ and maxima $\bar{v}_{i}$ on $X \cap C$.
    Using \eqref{eq:proof-x-contain-lev}, combined with \eqref{eq:primer},
    we thus have
    \(
    \nrm{v_{\trs(\ss)}}_{\infty}
    \leq
    \max \{  |\underline{v}_i| , |\bar{v}_i|\}_{i \in [\nModes]}
    \nfd \bar{v}
    \) for all $\ss \in \N$,
    as required.
\end{proof}
\begin{lem}[Parametric stability] \label{lem:parametric-stability}
    If \cref{asm:regularity} is satisfied, then the following statements hold:
    \begin{statements}
        \item \label{stat:continuous-value}
         the optimal value $\Vop(p)$ defined by \eqref{eq:parametric},
        is continuous at $\p$ relative to $\simplex_{\nModes}$.
        \item \label{stat:continuous-solution}
         For any $\ph_\ss \to \p$, and for any $\xx_{\ss} \in \Xop(\ph_\ss)$, 
        $\{\xx_{\ss}\}_{\ss \in \N}$ is bounded and all its cluster points lie in 
        $\Xop(\p)$. 
    \end{statements}
    
\end{lem}
\begin{proof}
If $L$ is continuous, then $V(x,p)$ can be written as
the composition $V \equiv g \circ F$ of the \ac{lsc} function
$g: \Re^{2 \nModes} \to \Re: (y,z) \mapsto \inprod{y,z} + \delta_{\simplex_{\nModes}}(p)$,
and $F: \Re^{n \nModes} \to \Re^{2\nModes}: (x, p) \mapsto (L(x), p)$. 
By \cite[Ex. 1.40(a)]{rockafellar_VariationalAnalysis_1998}, 
this implies that $V$ is \ac{lsc}, and so is $(x,p) \mapsto V(x,p) + \delta_X(x)$.
Moreover, by \cref{lem:uniform-level-bounded}, it is level-bounded in $x$ locally uniformly in $p$.
Furthermore, 
$p \mapsto V(\xx,p)$ is continuous relative to $\simplex_{\nModes}$ for all fixed $\xx \in X$. 
Thus, \cite[Thm. 1.17(b),(c)]{rockafellar_VariationalAnalysis_1998} applies, 
translating directly to \cref{stat:continuous-value,stat:continuous-solution}.
\end{proof}
\newcommand{\ourseq}{\inprod{\ph_\ss, L(\xx_{\trs(\ss)})}}
\begin{cor} \label{cor:convergence-vp}
	Let $\{\xx_{\trs(\ss)}\}_{\ss \in \N}$ be generated by
	\eqref{eq:primer}
	and let $\{\ph_\ss \in \simplex_{\nModes}\}_{\ss \in \N}$ be some sequence
	with $\ph_{\ss} \to \p$.
	Then,
	\[
		\lim_{\ss \to \infty} \ourseq = \Vop(\p)
	\]
\end{cor}

\begin{proof}
	By definition of $\Vop$, we have
	$\ourseq \geq \Vop(\ph_{\ss})$, and
	by \cref{lem:parametric-stability},
	$
		\lim_{\ss \to \infty} \Vop(\ph_{\ss}) \to \Vop(\p).
	$
	Therefore,
	\begin{equation} \label{eq:proof-app-liminf}
		\liminf_{\ss \to \infty} \ourseq \geq \Vop(\p).
	\end{equation}

	On the other hand,
	since the sequence $\{ x_{\trs(\ss)} \in X \}_{\ss}$ is bounded,
	and $L$ is continuous on $X$, $\ourseq$ has at least one
	cluster point and
	$
		\limsup_{\ss \to \infty} \ourseq < \infty.
	$
	Assume then, for the sake of contradiction, that there exists a cluster point
	$\bar{V} = \limsup_{\ss \to \infty} \ourseq > \Vop(\p)$.
    Since $\ph_{\ss} \to \p$, this implies, by continuity of $L$, that there must exist a limit point
	$\xx \notin \Xop(\p)$ of $\{ \xx_{\trs(\ss)} \}_{\ss}$, contradicting
	\cref{lem:parametric-stability}. We conclude that
	\begin{equation} \label{eq:proof-app-limsup}
        \limsup_{\ss \to \infty} \ourseq \leq \Vop(\p).
	\end{equation}
	Combining \eqref{eq:proof-app-liminf} and \eqref{eq:proof-app-limsup}
	completes the proof.
\end{proof}

    \subsection{Comparison with the Hoeffding bound} \label{sec:comparison-hoeffding}
We consider another instance of the example set-up from \cref{sec:numerical}, 
and compare \cadro{} using the 
the Hoeffding bound (\cref{prop:hoeffding})
and the 
ordered mean bound (\cref{prop:orm})
.

\Cref{fig:results-hoeffding} shows the 
cost estimate $\Vest_{\ss}$ and the out-of-sample
cost $V(\xest_{\ss}, \p)$ for the $\tvdro{}$ method and 
the aforementioned versions of \cadro.
We note that the radius of the ambiguity set for $\tvdro$
is computed using the 
Bretagnolle-Huber-Carol inequality \cite[Prop. A.6.6]{vandervaart_WeakConvergenceEmpirical_2000}
with slightly improved constants.
As this result is based on the same Hoeffding-type inequality as \cref{prop:hoeffding},
The apparent performance gains of \cadro{} with the Hoeffding bound
are thus to be attributed primarily to 
the geometry of the ambiguity set.
However, unlike divergence-based ambiguity sets, 
which rely on concentration inequalities to bound deviations of the 
distribution from the empirical mean, \eqref{eq:ambiguity-shape} does not require the 
use of concentration inequalities. Rather, any high-confidence upper bound on the 
mean of a scalar random variable satisfying the conditions of \cref{lem:consistency-conditions} may be used, 
allowing the use of more sophisticated approaches (e.g., \cref{prop:orm}).
This results in the improvements visible in \cref{fig:results-hoeffding},
without requiring alterations 
to the \ac{DRO} method itself.

\begin{figure}[ht!]
    {\footnotesize
    \centering
    \begin{minipage}{0.48\linewidth}
        \centering
        \includegraphics{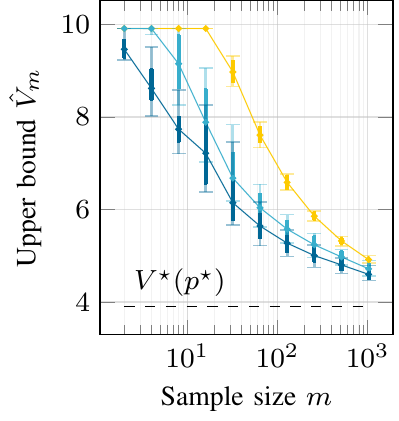}   
    \end{minipage}\hfill
    \begin{minipage}{0.48\linewidth}
        \centering
        \includegraphics{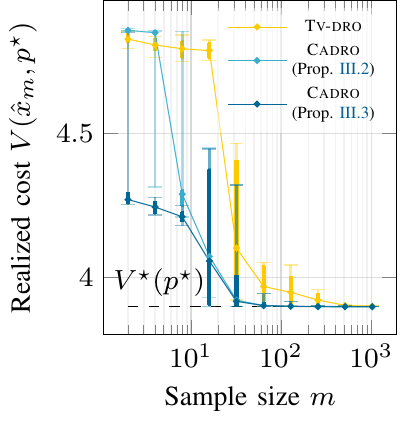}   
    \end{minipage}
    }
    \caption{
        Results for a problem instance as described in \Cref{sec:numerical}. 
        (left): 
         The cost estimates $\Vest_\ss$
        satisfying \eqref{eq:guarantee-coverage} and \eqref{eq:guarantee-consistency};
        (right): 
        True out of sample cost $V(\xest_\ss, \p)$.
        The points indicate the sample mean, 
        the solid errorbars indicate the empirical 0.95 (upper and lower) quantiles
        and the semi-transparent errorbars indicate the largest and smallest values over 100 independent runs. 
    }
    \label{fig:results-hoeffding}
\end{figure}
\end{appendix}
\fi

\end{document}